\documentclass[a4paper,leqno,11pt]{amsart}
\usepackage{amsfonts,amssymb,verbatim,amsmath,amsthm,latexsym,textcomp,amscd}
\usepackage{latexsym,amsfonts,amssymb,epsfig,verbatim}
\usepackage{amsmath,amsthm,amssymb,latexsym,graphics,textcomp}
\usepackage{graphicx}
\usepackage{color}
\usepackage{url}
\usepackage{enumerate}
\usepackage[mathscr]{euscript}
\usepackage{tikz}
\usepackage{tikz-cd}
\usepackage{dsfont}
\usetikzlibrary{matrix}
\usepackage{sseq}

\usepackage{accents}
\newlength{\dhatheight}

\usepackage{hyperref}

\input xy
\xyoption{all}

\renewcommand{\epsilon}{\varepsilon}
\renewcommand{\rho}{\varrho}
\renewcommand{\phi}{\varphi}

\newcommand{\bM}{\mathbb{M}}

\newcommand{\uM}{\underline{M}}

\newcommand{\Z}{\mathbb Z}

\newcommand{\CC}{\mathcal{C}}

\newcommand{\bF}{\mathbb{F}}
\newcommand{\bR}{\mathbb{R}}

\newcommand{\bFp}{\bF_p}
\newcommand{\uFp}{\underline{\bFp}}

\newcommand{\uZ}{\underline{\mathbb{Z}}}

\newcommand{\Id}{\operatorname{Id}}
\newcommand{\Burn}{\sf{Burn}}

\newcommand{\Sp}{\sf{Sp}}

\newcommand{\res}{\operatorname{res}}

\newcommand{\Rep}{\operatorname{Rep}}

\newcommand{\Conf}{\operatorname{Conf}}
\newcommand{\bs}{\bigstar}
\newcommand{\bsc}{\bs_\xi}

\newcommand{\tH}{\tilde{H}}
\newcommand{\tHbs}{\tilde{H}^\bigstar_G}

\newenvironment{myeq}[1][]
{\stepcounter{theorem}\begin{equation}\tag{\thetheorem}{#1}}
{\end{equation}}

     \theoremstyle{plain}
\newtheorem{theorem}{Theorem}[section]
\newtheorem{thm}[theorem]{Theorem}
\newtheorem{lemma}[theorem]{Lemma}
\newtheorem{cor}[theorem]{Corollary}
\newtheorem{prop}[theorem]{Proposition}

\newtheorem{notation}[theorem]{Notation}

\newtheorem{example}[theorem]{Example}

\theoremstyle{definition}
\newtheorem{defn}[theorem]{Definition}

\theoremstyle{remark}
\newtheorem{remark}[theorem]{Remark}

\begin{document}

\title{Revisiting the Nandakumar-Ramana Rao Conjecture}
\author{Surojit Ghosh}
\address{Department of Mathematics, Indian Institute of Technology, Roorkee, Uttarakhand-247667, India}
\email{surojit.ghosh@ma.iitr.ac.in; surojitghosh89@gmail.com}

\author{Ankit Kumar}
\address{Department of Mathematics, Indian Institute of Technology, Roorkee, Uttarakhand-247667, India}
\email{ankit\_k@ma.iitr.ac.in}
\date{\today}

\begin{abstract}
We reprove the generalized Nandakumar-Ramana Rao conjecture for the prime case using representation ring-graded Bredon cohomology. Our approach relies solely on the $RO(C_p)$-graded cohomology of configuration spaces,  viewed as a module over the $RO(C_p)$-graded Bredon cohomology of a point.
\end{abstract}
\maketitle

\section{Introduction}
Partitioning convex bodies in Euclidean spaces into convex pieces of equal measure is a long-standing problem in combinatorial geometry. This area of research originated with the celebrated \emph{Ham-Sandwich Theorem}, which states that in \(\mathbb{R}^m\), it is possible to bisect \(m\) masses using a single hyperplane.

In 2006, Nandakumar and Ramana Rao proposed a geometric conjecture, asking whether, \emph{for any convex shape in the plane and any positive integer \(n\), there exists a partition of the shape into \(n\) convex pieces such that all pieces have equal area and equal perimeter}. They provided a positive answer to the conjecture for \(n = 2^k\) in \cite{Nan12}. Later, Bárány, Blagojević, and Szűcs verified the case for \(n = 3\) in \cite{Bar10}. It is claimed in \cite{AAK18} that the conjecture for general \(n \geq 2\) has been resolved. Notably, all these advances have been made for convex bodies in \(\mathbb{R}^2\).

A generalized version of the Nandakumar-Ramana Rao conjecture can be formulated as follows:

\begin{quote}
    Let \(K\) be a \(d\)-dimensional convex body in \(\mathbb{R}^d\), let \(\mu\) be an absolutely continuous measure on \(\mathbb{R}^d\), let \(n \geq 2\) be a natural number, and let \(\phi_1, \ldots, \phi_{d-1}\) be \(d-1\) continuous functions on the metric space of \(d\)-dimensional convex bodies in \(\mathbb{R}^d\). Then, there exists a partition of \(\mathbb{R}^d\) into \(n\) convex pieces \(P_1, \ldots, P_n\) such that  
    \[
    \mu(P_1 \cap K) = \cdots = \mu(P_n \cap K)
    \]
    and 
    \[
    \phi_i(P_1 \cap K) = \cdots = \phi_i(P_n \cap K)
    \]
    for every \(1 \leq i \leq d-1\).
\end{quote}

This problem can be reduced to determining the non-existence of a \(\Sigma_n\)-equivariant map from \(\Conf_{n}(\mathbb{R}^d)\) to a certain sphere, where \(\Conf_{n}(\mathbb{R}^d)\) denotes the configuration space of \(n\) distinct labeled points in \(\mathbb{R}^d\).

The generalized conjecture for \(n\), where \(n\) is a prime power, was established for any dimension by Blagojević and Ziegler in \cite{BZ14} using equivariant obstruction theory. Independently, the same result was obtained in \cite{KHA14} based on a cell decomposition of the one-point compactification of the configuration space.
In \cite{BLZ}, the authors reprove the generalized version for arbitrary dimensions and prime power \(n\) using the Fadell-Husseini index of the configuration space.

In this paper, we present an alternative proof of the generalized Nandakumar-Ramana Rao conjecture for the case where the number of parts \( n \) is a prime \( p \). Our approach establishes the non-existence of a \( C_p \)-equivariant map from the configuration space \( \Conf_{p}(\mathbb{R}^d) \) to a certain sphere in the \( C_p \)-representation. This method differs significantly from those previously described in the literature.

Specifically, we employ the module structure of the \( RO(C_p) \)-graded Bredon cohomology of configuration spaces, viewed as a module over the \( RO(C_p) \)-graded cohomology of a point. 

\begin{notation}
   \begin{enumerate}
       \item  The representation ring $RO(C_p)$ is generated by the trivial representation $1$, and the $2$-dimension representations $\xi^k$ given by the rotation by the angle $\frac{2\pi k}{p}$ for $k=1, \cdots, \frac{p-1}{2}$. 
\item We use \( \bs \) to denote the \( RO(G) \)-grading, and \( * \) for the usual integer grading.

    \item For \( G = C_p \), with \( p \) an odd prime, we use \( \bsc \) to denote the \( m + n\xi \) grading, where \( m, n \in \mathbb{Z} \).
\end{enumerate}
\end{notation}

\paragraph{\textbf{Acknowledgement.}} We would like to thank the referee for their detailed and pertinent comments, which have helped improve the clarity and presentation of this manuscript. The research of the second author (AK) is supported by the University Grants Commission, India.

\section{Preliminaries on Bredon cohomology}

Ordinary cohomology theories are defined for abelian groups and are represented by spectra with homotopy concentrated in degree 0. In the equivariant setting, the analogous role is played by Mackey functors. In this section, we briefly recall their definition and relate them to equivariant cohomology (see \cite{May96} for details).

The \emph{Burnside category} $\Burn_G$ is the category whose objects are finite $G$-sets. For objects $S$ and $T$, the morphism set is given by the group completion of the hom-set of spans between $S$ and $T$ in the category of finite $G$-sets.  

\begin{defn}
A \emph{Mackey functor} is a functor $\uM: \Burn_G^{op} \to \mathsf{Ab}$ from the opposite of the Burnside category to the category of abelian groups.
\end{defn}

In this article, we fix $G = C_p$, the cyclic group of prime order $p$. 

\begin{example}
    For a $G$-module $M$, we define a Mackey functor $\uM$ by the formula $\uM(G/H) = M^H$. In particular, we consider the Mackey functors $\uZ$ and $\uFp$, corresponding to the trivial $G$-modules $\Z$ and $\bFp$, respectively.
\end{example}

For a real $G$-representation $V$ equipped with a $G$-invariant inner product, we define 
\[
S(V) := \{ v \in V \mid \langle v, v \rangle = 1 \},
\]
the unit sphere in $V$, and $S^V$, the one-point compactification of $V$. In equivariant homotopy theory, the $V$-fold suspension map 
\[
X \mapsto S^V \wedge X
\]
is invertible. We refer to \cite{MM02} for constructing the equivariant stable homotopy category, denoted by $\Sp^G$, which is the homotopy category of $G$-spectra. 

Since the $V$-fold suspension map is invertible, one can define $S^\alpha$ for $\alpha \in RO(G)$, the Grothendieck group completion of the monoid of irreducible representations of $G$. This construction induces an $RO(G)$-grading on homotopy groups. Specifically, for a based $G$-space $X$, the ordinary Bredon cohomology with coefficients in the Mackey functor $\uM$ at grading $\alpha \in RO(G)$ is given by:  
\[
\tilde{H}^\alpha_G(X; \uM) := {\Sp}^G(X, S^\alpha \wedge H\uM),
\]
here $H\uM$ denotes the  Eilenberg--Mac Lane spectrum (cf. \cite{HHR16}) for the  Mackey functor $\uM$ defined as  
\[
\underline{\pi_n}(H\uM) = 
\begin{cases} 
\uM, & n = 0, \\
0, & n \neq 0.
\end{cases}
\]

\begin{remark}\label{cpgrad}
   For $G=C_p$, one observes that it suffices to consider the grading of the form $m+n \xi$ with $m,n \in \Z$ as $S^\alpha \wedge H\uFp \simeq H\uFp$ whenever both the dimensions of $\alpha$ and $\alpha^{C_p}$ are zero (See \cite[Appendix B]{FL04}). 
\end{remark}

Let $V$ be a $G$-representation. We denote by $a_V$ the inclusion of the fixed points,  
\[
a_V: S^0 \to S^V.
\]
For a ring spectrum $X$ with a $G$-action, we abuse notation and also denote by $a_V$ its image under the map $S^0 \to X$.

If the representation $V$ contains the trivial representation as a summand, then $a_V = 0$.  
Moreover, for any two $G$-representations $V$ and $W$, we have the relation  
\[
a_V a_W = a_{V \oplus W}.
\]

\noindent See \cite[Definition 3.11]{HHR16} for further details.

For an orientable $G$-representation $V: G \to SO(V)$, a choice of orientation induces an isomorphism  
\[
\tilde{H}^{\dim(V)}_G(S^V; \uZ) \cong \Z.
\]
The Thom space of the equivariant bundle $V \to G/G$ is $S^V$. In particular, the restriction map  
\[
\tilde{H}^{\dim(V)}_G(S^V; \uZ) \to \tilde{H}^{\dim(V)}_e(S^{\dim(V)}; \Z)
\]
is an isomorphism.

Utilizing the above isomorphism, for an orientable $G$-representation $V$, we define the orientation class $u_V$ as  
\[
u_V \in \tilde{H}^{V-\dim(V)}_G(S^0 ; \uZ),
\]
the generator that maps to $1$ under the restriction isomorphism. The orientation class satisfies the relations:  
\[
u_{V \oplus 1} = u_V, \quad u_V \cdot u_W = u_{V \oplus W}.
\]
\noindent See \cite[Definition 3.12]{HHR16} for more details.
  
We use the same notations for the image of the classes $a_V$ and $u_V$ in $\tilde{H}^{\bs}_G(S^0; \uFp)$ under the spectrum map $H \uZ \to H\uFp.$  Whenever needed, we shall write \( \mathbb{M}_p \) to refer to the restricted \( RO(C_p) \)-graded cohomology ring \( \tilde{H}^{\bsc}_{C_p}(S^0; \uFp) \).

\section{Some $\mathbb{M}_{p}$-modules}
In this section, we describe the $\bM_p$-module structure on the $RO(C_p)$-graded Bredon cohomology of the universal space $EC_p$ and the configuration space $\Conf_p(\bR^d)$. Observe that the collapsing map $X_+ \to S^0$ induces a $\bM_p$-module structure on $\tilde{H}^\bs_{C_p}(X_+; \uFp)$. For any free $C_p$-space $X$, the Bredon cohomology of $X$ can be analyzed using a spectral sequence, which in the case of $X = EC_p$, corresponds to the homotopy fixed point spectral sequence.

We start with a $C_p$-CW structure on $X$ as $X=\cup_s X^{(s)}$. If $X$ is free, the cells are of the type $G/e \times \mathbb{D}^k$. We also note that $X/G$ has a CW complex structure with associated filtration $X^{(s)}/G$. Thus, we have 
$$X^{(s)}/ X^{(s-1)}\cong \bigvee_{e \in I(s)}G_+ \wedge S^s.$$

Hence following \cite{BG21}, we have an $RO(C_p)$-graded spectral sequence as follows:

\begin{prop} \label{HFPSS} There is a spectral sequence 
$$E^{s,t}_2(\alpha) = H^s(X/G; {\pi}_t(S^{-\dim(\alpha)}\wedge H\bFp )) \Rightarrow H^{s-t-\alpha}_G(X, \uFp). $$
with boundary $d_r : E_r^{s,t}(\alpha) \to E_r^{s+r,t-r+1}(\alpha).$ The spectral sequences assemble together for various $\alpha$ into a multiplicative $RO(C_p)$-graded spectral sequence 
$$E^{s,\alpha}_2 = H^s(X/G; \pi_0 (S^{-\dim(\alpha)}\wedge H\bFp)) \Rightarrow  H^{s-\alpha}_G(X, \uFp)$$ 
where $s\in \Z$ and $\alpha \in RO(G)$. 
\end{prop}

The multiplicative structure on the spectral sequence in Proposition \ref{HFPSS} implies that everything is a product of the form $H^s(X/G) \otimes \pi_0 (S^{-\dim{\alpha}} \wedge H\bFp)$. Now $\pi_0 (S^{-\dim\alpha} \wedge H\bFp)$ fits together as a ring $\bigotimes\limits_{\eta \in \widehat{G}} \bFp[u_\eta^\pm]$, where $\widehat{G}$ denotes the set of characters of $G$. Thus we obtain 

\begin{cor}{\label{uni}}
    The Bredon cohomology of $E{C_p}$ is given by 
 
  \[
\tilde{H}^{\bs}_{C_p}({EC_p}_+; \uFp) \cong \frac{\bFp[x, y, u_\eta^{\pm} \colon \eta \in \widehat{G}\setminus \{1\}]}{(x^2)},
\] 
with $|x|=1$ and $|y|=2$. In particular, \[
\tilde{H}^{\bsc}_{C_p}({EC_p}_+; \uFp) \cong \frac{\bFp[x, y, u_\xi^{\pm}]}{(x^2)}.
\]
\end{cor}
\begin{remark}
  We identify the class $a_\xi$ with $yu_\xi$ and define $\kappa_\xi$ as $xu_\xi$, respectively. For a detailed discussion, see \cite[page 7]{BG21}.  
\end{remark}
Since $C_p$ acts on $\Conf_{p}(\mathbb{R^d})$ freely, one can use the similar technique used in Corollary $\ref{uni}$ to calculate the Bredon cohomology of configuration space:

\begin{cor}
    The $RO(C_p)$-graded Bredon cohomology of the ordered configuration space $\Conf_p(\bR^d)$, is 
    $$H^\bigstar_{G}(\Conf_p(\bR^d);\uFp) \cong H^{\ast}(\widehat{B}_p(\bR^d);\bFp)\otimes \bigotimes_{\eta \in \widehat{G} \setminus \{1 \}} \bFp [u_\eta^\pm].$$
    Here $\widehat{B}_p(\bR^d)$ denotes the orbit space $\Conf_p(\bR^d)/C_p.$
\end{cor}

At this point, we describe the ring structure of the cohomology of a point $\bM_p$ using the pullback (``Tate square") \cite{GM95}, which is induced from the isotropy separation sequence ${EC_p}_{+} \to S^0 \to \widetilde{EC_p}$ as follows: 
\begin{myeq}\label{tate}
\xymatrix{H\uFp \ar[d]_{q} \ar[r] &  \widetilde{EC_p} \wedge H\uFp \ar[d] \\ F({EC_p}_+, H\uFp) \ar[r] &\widetilde{EC_p} \wedge F({EC_p}_+, H\uFp)}
\end{myeq}
Note that one model for $\widetilde{EC_p}$ can be given by the colimit of the directed sequence $(S^{n\xi} \stackrel{a_\xi}{\to} S^{(n+1)\xi})_{n\ge 0}$. Therefore, computing $\pi_{\bsc} (\widetilde{EC_p}\wedge X)$ requires determining the $\Z$-graded homotopy groups and tensoring with $\bFp[a_\xi^\pm]$.   Thus we compute the restricted $RO(C_p)$-graded Bredon cohomology $\tilde{H}^{\bsc}_{C_p}(S^0; \uFp) \cong \pi_{-\bsc}(H\uFp)$ from the \v{C}ech-cochain 
\[
0 \to \frac{\bFp[a_\xi, \kappa_\xi, u_\xi^{\pm}]}{(\kappa_{\xi}^2)} \oplus \frac{\bFp[a_\xi^{\pm}, \kappa_\xi, u_\xi]}{(\kappa_{\xi}^2)} \to \frac{\bFp[a_\xi^\pm, \kappa_\xi, u_\xi^{\pm}]}{(\kappa_{\xi}^2)} \to 0.
\]

\begin{prop}[\cite{BG21}, Proposition 3.6]
    For odd $p$ prime, 
    \[
     \tilde{H}^{\bsc}_{C_{p}}(S^{0};\uFp)=\bFp[a_{\xi},\kappa_{\xi},u_{\xi}]/(\kappa^{2}_{\xi}) \bigoplus_{j,k\geq 1} \bFp\{\Sigma^{-1}\frac{\kappa^{\epsilon}_{\xi}}{a^{j}_{\xi}u^{k}_{\xi}}\}
    \]
    where $\epsilon\in \{0,1\}$.
\end{prop}

Note that both  $\tilde{H}^{\bs}_{C_{p}}(S^{0};\uFp)$ and $\bM_p= \tilde{H}^{\bsc}_{C_{p}}(S^{0};\uFp)$ are rings and for any $C_p$-space $X$, the cohomology $\tilde{H}^{\bs}_{C_{p}}(X_+;\uFp)$ carries the structure of a module over both rings.  The $\mathbb{M}_{p}$-module structure on $\tilde{H}^\bs_{C_p}({EC_p}_+; \uFp)$ is given by the map $q$ (in \eqref{tate}) as follows.
\begin{prop}[\cite{BG21}, Proposition 3.10]
     For the odd prime $p$, the action of $u_{\xi},a_{\xi},\kappa_{\xi} \in \bM_p$ on $\tilde{H}^{\bigstar}_{C_p}({EC_{p}}_{+};\uFp)$  is given by the multiplication of classes $u_{\xi}, yu_{\xi}$ and $xu_{\xi}$ respectively.
 \end{prop}

Note that $C_p$ acts on $\Conf_p(\bR^d)$ freely. Hence the Borel construction $EC_p \times_{C_p} \Conf_p(\bR^d) \simeq \widehat{B}_p(\bR^d).$ To compute the cohomology $H^\ast(\widehat{B}_p(\bR^d); \bFp)$, we consider the Borel fibration 
\[
\Conf_p(\bR^d) \to \widehat{B}_p(\bR^d) \stackrel{\pi}{\to} BC_p
\]
 and the associated Serre spectral sequence with the $E_2$-page given by
 \[
 E_2^{r,s} = H^r(C_p; H^s(\Conf_p(\bR^d); \bFp)) \Rightarrow H^{r+s}(\widehat{B}_p(\bR^d); \bFp)
 \]

Here, we shall not compute the above spectral sequence, which is a tedious job and completely computed in  \cite{BLZ}. Rather, we observe an immediate fact that the map $\pi^\ast \colon H^j(BC_p; \bFp) \to H^j(\widehat{B}_p(\bR^d); \bFp)$ is a non-trivial map for $j \le (p-1)(d-1).$ In particular, $\pi^\ast(x)$ and $\pi^\ast(y)$ are non-trivial, where $H^\ast(BC_p; \bFp)\cong \frac{\bFp[x, y]}{(x^2)}$ with $|x|=1$ and $|y|=2.$

Since $\Conf_p(\bR^d)$ is a free $C_p$-space, hence we have  based $C_p$-maps $\Conf_p(\bR^d)_+ \to {EC_p}_+ \to S^0$. Thus $\tilde{H}^\bigstar_{G}(\Conf_p(\bR^d)_+;\uFp)$ is an $\mathbb{M}_p$-module. Moreover, we have the following commutative diagram
\[  
\xymatrix{\bM_p\ar[rr] \ar[drr] && \tilde{H}^\bs_{C_p}({EC_p}_+; \uFp) \cong H^{\ast}(BC_p)\otimes \bigotimes\limits_{\eta \in \widehat{G} \setminus \{1 \}} \bFp [u_\eta^\pm]\ar[d]^{\pi^\ast \otimes \Id} \\ & & \tilde{H}^\bs_{C_p}({\Conf_p(\bR^d)}_+; \uFp) \cong H^{\ast}(\widehat{B}_p(\bR^d))\otimes \bigotimes\limits_{\eta \in \widehat{G} \setminus \{1 \}} \bFp [u_\eta^\pm]}.
\]
Thus we obtain
\begin{prop}\label{action}
The action of $u_\xi \in \bM_p$ on $\tilde{H}^\bs_{C_p}({\Conf_p(\bR^d)}_+; \uFp)$ is given by the multiplication of the corresponding element $u_{\xi}$. The action of $a_\xi, \kappa_\xi$ equal multiplications by $\pi^\ast(y)u_\xi$  and $\pi^\ast(x)u_\xi$ respectively.
\end{prop}

\section{Connection with equivariant maps and the proof}

\subsection{The equivariant map} The general idea of the ``Configuration space/Test map” scheme says an affirmative answer to a topological combinatorics problem is equivalent to the non-existence of some equivariant map from a specific configuration space to the test space.

In \cite{KHA14}, this scheme is applied to investigate the validity of the Nandakumar-Ramana Rao conjecture and its generalized version, as follows:
\begin{prop}
    Let $d \geq 1, n \geq 2$ be integers. Let $K\subset \mathbb{R}^{d}$ be a $d$-dimensional convex body. If there is no $\Sigma_{n}$-equivariant map of the form 
    \[
    \Conf_{n}(\mathbb{R}^{d}) \to S(W_{n}^{\oplus(d-1)})
    \]
    then the Nandakumar-Ramana Rao conjecture has a solution.
\end{prop}
In the above proposition, the action of $\Sigma_{n}$ on the standard $\Sigma_n$-representation $W_{n}=\{(v_1,\cdots,v_n)\in \mathbb{R}^n|\sum v_i=0\}$ is given by permuting the coordinates. Since $C_{n}\subset\Sigma_{n}$, the above spaces inherit the restricted $C_{n}$-action. Consequently, the non-existence of a $C_{n}$-equivariant map is sufficient to affirm the conjecture.

\subsection{Bredon cohomology of representation spheres}
Let $V$ be a $G$-representation. We define  $\CC_{[\supseteq V]}$ -- the collection of $G$-representations $W$ such that $V$ is a sub-representation of $W$. We denote by $W \ominus V$, the orthogonal complement of $V$ in $W.$
\begin{lemma}\label{Cohsphere}
    For $V \in \Rep(C_{p}),$ $\tilde{H}^W_{C_p}(S(V)_+; \uFp)=0$, where $W \in \CC_{[\supseteq V]}.$
\end{lemma}

\begin{proof}
If $V$ is reducible with a summand $L$, that is,  $V \in \CC_{[\supseteq L]},$  consider the following $G$-cofiber sequence
\[
S(V\ominus L)_+ \to S(V)_+ \to S^{V\ominus L}\wedge S(L)_+
\]
which induces the long exact sequence in Bredon cohomology:
\[
\cdots \tH^{\bs -(V\ominus L)}_{G}(S(L)_+; \uFp) \to \tHbs(S(V)_+; \uFp) \to \tHbs(S(V\ominus L)_+; \uFp) \to \cdots
\]
By induction on the dimension of the representation $V$ and using $\bs=W \in \CC_{[\supseteq V]}$, it follows that both corner groups are trivial, leading to the desired result.

Therefore, we are left with the case when $V$ is irreducible, that is, $V= \xi^k$ for some $k =1, \cdots , \frac{p-1}{2}.$ Now consider the $C_p$-cofiber sequence (in $C_p$-spectra)
\begin{myeq}\label{sxi}
 C_{p}/e_{+} \stackrel{1-g^k}{\to} C_{p}/e_{+} \to S(\xi^k)_{+},
\end{myeq}
where we fix $g$ to be a generator of $C_p.$ The associated long exact sequence in Bredon cohomology for \eqref{sxi}, along with the identification $\tilde{H}^\alpha_{C_p}({C_p/e}_+; \uFp) \cong \tilde{H}^{\dim(\alpha)} (S^0; \bFp)$ completes the proof.
\end{proof}

\begin{prop}
    For a cyclic group $G$ with order having two distinct prime divisors, the Euler class $a_{\bar{\rho}}$ vanishes. Here, $\bar{\rho}$ denotes the reduced regular representation of $G.$
\end{prop}

\begin{proof}
Let \( G = C_n \) be the cyclic group of order \( n \), and let \( p \neq q \) be two distinct prime divisors of \( n \). We construct two subrepresentations \( V_p \) and \( V_q \) of the reduced regular representation \( \bar{\rho} \) such that
\[
\gcd\left( \frac{|G|}{|G_{V_p}|}, \frac{|G|}{|G_{V_q}|} \right) = 1,
\]
where \( G_V \) denotes the stabilizer subgroup of the representation \( V \).

To achieve this, let \( \eta \) be a non-trivial one-dimensional complex \( G \)-representation induced by a group homomorphism \( \eta \colon G \to S^1 \), with kernel \( H = \ker(\eta) \). The corresponding representation sphere \( S^\eta \) admits a \( G \)-CW complex structure of the form
\[
S^0 \cup (G/H_+ \wedge e^1) \cup (G/H_+ \wedge e^2).
\]
The associated equivariant cellular chain complex with constant coefficients \( \uZ \) is
\[
\mathbb{Z} \xrightarrow{0} \mathbb{Z} \xrightarrow{|G/H|} \mathbb{Z},
\]
which gives rise to the isomorphism
\[
\tilde{H}^\eta_G(S^0; \uZ) \cong \mathbb{Z}/|G/H|\mathbb{Z}.
\]
In particular, this implies that \( |G/H| \cdot a_\eta = 0 \).

We now choose the subrepresentations \( V_p = \xi^{n/p} \) and \( V_q = \xi^{n/q} \), where \( \xi^d \) denotes the one-dimensional complex representation of \( G \) corresponding to rotation by angle \( \frac{2\pi d}{n} \). Since the indices \( \frac{|G|}{|G_{V_p}|} \) and \( \frac{|G|}{|G_{V_q}|} \) are coprime, it follows that the Euler class \( a_{\bar{\rho}} = 0 \).
\end{proof}

\begin{remark}
This proposition sheds light on the fact that the non-existence of maps from any free $C_n$-space $X$  to the sphere $S(k\bar{\rho})$ is not possible if $n$ is not a prime power. Although we are not considering the non-prime power case, one can use obstruction theory for Bredon cohomology for proof.
\end{remark}

\subsection{The proof}
Observe that the inclusion ${EC_p}^{((d-1)(p-1))} \subseteq EC_p$ induces an isomorphism in $\tilde{H}^\ast_{C_p}(-, \uFp)$ for $\ast \leq (d-1)(p-1)-1$ and is injective for $\ast = (d-1)(p-1).$ Since this is an inclusion of free $C_p$-spaces the result also holds for $\tilde{H}^\alpha_{C_p}(-; \uFp)$ where $\alpha \in RO(C_p)$ with $\dim \alpha \leq (d-1)(p-1).$ In particular,  observe for $1 \in \tilde{H}^0_{C_p}({\Conf_p(\bR^d)}_+; \uFp),$ Proposition \ref{action} yields 
 $$a_{\frac{(d-1)(p-1)}{2} \xi} . 1= {\pi^\ast(y)}^{\frac{(d-1)(p-1)}{2}}u_\xi^{\frac{(d-1)(p-1)}{2}} \neq 0,$$
thus, 
\begin{myeq} \label{acteg} 
0\neq a_{\frac{(d-1)(p-1)}{2} \xi} . 1 \in \tilde{H}^{\frac{(d-1)(p-1)}{2}\xi}_{C_p}({\Conf_p(\bR^d)}_+; \uFp).
\end{myeq}

\begin{thm}
    There does not exist any $\Sigma_p$-equivariant map from $\Conf_p(\bR^d)$ to $S(W_p^{\oplus (d-1)}).$
\end{thm}
  
\begin{proof}
Note that $\res_{C_p}(W_p)$ is the reduced regular representation $\bar{\rho}$ of $C_p$. Suppose on the contrary, there is a $C_p$-map $f: \Conf_p(\bR^d) \to S(\bar{\rho}^{\oplus (d-1)}).$ Then it induces a $\tilde{H}^\bs_{C_p}(S^0; \uFp)$-module map 
$$f^\bs: \tilde{H}^\bs_{C_p}(S(\bar{\rho}^{\oplus (d-1)})_+; \uFp) \to \tilde{H}^\bs_{C_p}({\Conf_p(\bR^d)}_+; \uFp).$$ 

Thus, we have a commutative diagram
\[
\xymatrix{\tilde{H}^0_{C_p}(S(\bar{\rho}^{\oplus (d-1)})_+; \uFp)\ar[d]_{a_{\frac{(p-1)(d-1)}{2} \xi}.} \ar[r]^{f^\bs} & \tilde{H}^0_{C_p}({\Conf_p(\bR^d)}_+; \uFp) \ar[d]^{a_{\frac{(p-1)(d-1)}{2} \xi}.} \\ \tilde{H}^{\frac{(p-1)(d-1)}{2} \xi}_{C_p}(S(\bar{\rho}^{\oplus (d-1)})_+; \uFp) \ar[r]_{f^\bs} & \tilde{H}^{\frac{(p-1)(d-1)}{2} \xi}_{C_p}({\Conf_p(\bR^d)}_+; \uFp)}
\]
So, using module structure, we have $$f^\bs(a_{\frac{(p-1)(d-1)}{2}\xi}.1)=a_{\frac{(p-1)(d-1)}{2}\xi}f^*(1).$$  Lemma \ref{Cohsphere} and Remark \ref{cpgrad} together imply that $a_{\frac{(p-1)(d-1)}{2} \xi}.1=0.$ Therefore, this contradicts \eqref{acteg}.
\end{proof}

\begin{remark}
For $p=2$, the proof follows the same structure as the odd prime case, with only minor adjustments required in the cohomology computations. As these changes are straightforward, we omit the details here.  The prime-power case we investigate in a future work.
\end{remark}

\end{document}